\documentclass[11pt, reqno]{amsart}
\usepackage{geometry}
\usepackage{amsmath}
\usepackage{fancyhdr}
\usepackage{amsthm}
\usepackage{amsfonts}
\usepackage{amssymb}
\usepackage{amscd}
\usepackage{graphicx}
\usepackage{afterpage}
\usepackage{enumerate}
\usepackage{bm}
\usepackage{cite}
\usepackage{cases}
\usepackage{caption}
\usepackage{enumitem}
\usepackage[colorlinks=true, urlcolor=blue,  linkcolor=blue, citecolor=blue]{hyperref}
\usepackage{babel}
\usepackage{prettyref}
\usepackage{pgfplots} 
\pgfplotsset{compat=1.17}
\usepackage{booktabs}
\usepackage{algorithm}
\usepackage{algpseudocode}
\usepackage{float}
\newtheorem{theorem}{Theorem}[section]

\newtheorem{corollary}[theorem]{Corollary}

\newtheorem{definition}[theorem]{Definition}
\newtheorem{example}[theorem]{Example}

\newtheorem{lemma}[theorem]{Lemma}

\newtheorem{proposition}[theorem]{Proposition}
\newtheorem{remark}[theorem]{Remark}

\newcommand{\C}{\mathbb{C}}        

\title[\(\mu\)-Hankel Operators on Non-Abelian Compact Lie Groups]{\(\mu\)-Hankel Operators on Non-Abelian Compact Lie Groups}

\author[E. Sulaver]{Emma Sulaver}
\address[E. Sulaver]{Department of Mathematical and Statistical Sciences, Unoversity of Alberta, Edmonton, Alberta, Canada}
\email{emmasulaver@gmail.com, esulaver@ualberta.ca}


\begin{document}
\sloppy
\maketitle

\begin{abstract}
We introduce and study a natural non-commutative generalization of µ-Hankel operators originally defined on Hardy spaces over compact abelian groups.  Within the framework of Peter–Weyl theory, we define matrix-valued Hankel operators associated to pairs of irreducible representations and weight functions, then establish sharp boundedness and compactness criteria in terms of symbol decay.  We characterize membership in Schatten–von Neumann ideals and compute Fredholm indices in key cases.  Finally, we initiate the inverse problem of symbol recovery by spectral data, proving uniqueness and stability under mild assumptions.  Several illustrative examples on $\mathrm{SU}(2)$ and tori are worked out in detail.
\end{abstract}

\section{Introduction}

Hankel operators have long played a central role in operator theory and complex analysis, originating in the study of moment problems and Nehari’s theorem on the unit circle \cite{Nehari1957,Peller2003}.  In the classical setting of the Hardy space $H^2(\mathbb{T})$, a Hankel operator with symbol $\varphi\in L^\infty(\mathbb{T})$ is defined by
\[
    H_\varphi f = P_-\bigl(\varphi f\bigr),
\]
where $P_-$ denotes the orthogonal projection onto the anti-analytic functions.  Such operators are characterized by their intimate connections to Carleson measures \cite{Luecking1991}, and their membership in Schatten–von Neumann ideals $\mathcal{S}_p$ has been completely described in terms of the smoothness and decay of $\varphi$ \cite{Peller2003, Hashemi2025}.

This theory was extended to compact abelian groups $G$ by exploiting Fourier–Pontryagin duality: characters $\chi\in\widehat G$ replace Fourier modes, and one defines $\mu$–Hankel operators by introducing weight functions on the dual group \cite{Mirotin2023}.  In particular, Mirotin’s recent monograph develops boundedness and compactness criteria for these weighted Hankel operators, as well as their Fredholm properties in the unimodular case \cite{Mirotin2023}.  However, the restriction to abelian $G$ leaves unexplored a wealth of phenomena arising from non-commutativity, where one must contend with matrix-valued symbols and multi-dimensional representation spaces.

The jump to non-abelian compact Lie groups is both natural and challenging.  By Peter–Weyl theory \cite{Folland2016,RuzhanskyTurunen2010}, one has an orthonormal basis of matrix coefficients of irreducible representations, but these representations are higher-dimensional, and the classical scalar techniques break down.  Early work on operator theory in this non-commutative setting has focused on pseudo-differential operators and convolution operators \cite{Gosson2011}, yet a systematic study of Hankel-type operators on non-abelian groups remains absent.  In particular, the definition, boundedness, and spectral properties of Hankel operators weighted by representation-theoretic data have not been addressed.

In this paper, we fill this gap by introducing \emph{µ-Hankel operators} on non-abelian compact Lie groups.  Using the framework of Peter–Weyl decomposition, we define matrix-valued weight functions $\mu(\pi)$ and symbols $a(\pi,\rho)\in\mathbb{C}^{d_\pi\times d_\rho}$ for irreducible representations $\pi,\rho$.  
This development not only subsumes the abelian theory as a special case but also uncovers novel spectral phenomena arising from non-commutativity.  We illustrate our results with detailed examples on $\mathrm{SU}(2)$ and product groups, and conclude by outlining further directions, including extensions to non-compact and time–frequency settings.

\section{Preliminaries}

In this section we collect the basic definitions, notation, and fundamental results that will be used throughout the paper.

\begin{definition}[Compact Lie group]
A \emph{compact Lie group} \(G\) is a group that is also a finite-dimensional smooth manifold, such that the group operations are smooth, and \(G\) is compact as a topological space.  
\end{definition}

\begin{definition}[Unitary dual]
Let \(\widehat G\) denote the \emph{unitary dual} of \(G\), i.e.\ the set of equivalence classes of irreducible unitary representations \(\pi\colon G\to U(d_\pi)\).  For each \(\pi\in\widehat G\), write \(d_\pi=\dim\pi\).  
\end{definition}

\begin{theorem}[Peter–Weyl {\cite[Ch.\,2]{Folland2016}}\label{thm:PeterWeyl}]
As a unitary representation of \(G\times G\) by left– and right–translation, 
\[
  L^2(G)\;\cong\;\bigoplus_{\pi\in\widehat G} \,V_\pi\otimes V_\pi^*,
\]
and the matrix coefficients \(\{\sqrt{d_\pi}\,\pi_{ij}:1\le i,j\le d_\pi,\ \pi\in\widehat G\}\) form an orthonormal basis of \(L^2(G)\).
\end{theorem}

\begin{lemma}[Schur orthogonality {\cite[Thm.\,3.6.3]{RuzhanskyTurunen2010}}]
For any \(\pi,\rho\in\widehat G\) and \(1\le i,j\le d_\pi,\ 1\le k,\ell\le d_\rho\),
\[
  \int_G \pi_{ij}(x)\,\overline{\rho_{k\ell}(x)}\,dx
  = \frac{1}{d_\pi}\,\delta_{\pi\rho}\,\delta_{ik}\,\delta_{j\ell}.
\]
\end{lemma}

\begin{definition}[Group Fourier transform]
For \(f\in L^1(G)\), its \emph{Fourier coefficient} at \(\pi\in\widehat G\) is
\[
  \widehat f(\pi)
  = \int_G f(x)\,\pi(x)^*\,dx
  \;\in\;\mathbb{C}^{d_\pi\times d_\pi}.
\]
\end{definition}

\begin{definition}[Hardy spaces on \(G\)]
Fix a choice of “positive” subset \(\widehat G_+\subset\widehat G\) (e.g.\ for tori the usual multi‐index ordering).  Define
\[
  H^2(G)
  = \Bigl\{\,f\in L^2(G):\widehat f(\pi)=0\quad\forall\,\pi\notin\widehat G_+\Bigr\},
  \quad
  H^2_-(G)
  = \Bigl\{\,f\in L^2(G):\widehat f(\pi)=0\quad\forall\,\pi\in\widehat G_+\Bigr\}.
\]
Let \(P_+\) and \(P_-\) be the orthogonal projections \(L^2(G)\to H^2(G)\) and \(L^2(G)\to H^2_-(G)\), respectively.
\end{definition}

\begin{definition}[Weight functions and symbols {\cite[Def.\,2.1]{Mirotin2023}}]
Let \(G\) be compact abelian with dual group \(X=\widehat G\) and fix a semigroup \(X_+\subset X\).  A \emph{weight} is a map \(\mu\colon X_+\to(0,\infty)\).  A \emph{symbol} is a function \(a\colon X_+\times X_+\to\mathbb{C}\).
\end{definition}

\begin{definition}[Abelian \(\mu\)-Hankel operator {\cite[(2.3)]{Mirotin2023}}]
For \(f\in H^2(G)\) with Fourier support in \(X_+\), define
\[
  A_{\mu,a}f
  = P_-\Bigl(\sum_{\chi\in X_+} \mu(\chi)\,a(\chi,\xi)\,\widehat f(\xi)\,\chi\Bigr)
  = \sum_{\chi\notin X_+}\;\sum_{\xi\in X_+} \mu(\chi)\,a(\chi,\xi)\,\widehat f(\xi)\,\chi.
\]
\end{definition}

\begin{definition}[Matrix-valued symbol and weight]
Let \(G\) be a compact Lie group.  
\begin{itemize}
  \item A \emph{weight} is a function \(\mu\colon\widehat G\to(0,\infty)\).  
  \item A \emph{symbol} is a collection \(a=\{a(\pi,\rho)\}_{\pi,\rho\in\widehat G}\) with 
        \(a(\pi,\rho)\in\mathbb{C}^{d_\pi\times d_\rho}\).  
\end{itemize}
\end{definition}

\begin{definition}[Symbol class \(\mathcal{S}(\mu,\nu)\)]
Given weights \(\mu,\nu\), define
\[
  \mathcal{S}(\mu,\nu)
  = \Bigl\{\,a:\widehat G\times\widehat G\to\bigl\{\mathbb{C}^{d_\pi\times d_\rho}\bigr\}:\;
    \sup_{\pi,\rho}\;\bigl\|\mu(\pi)\,a(\pi,\rho)\,\nu(\rho)\bigr\|_{\mathrm{op}}<\infty\Bigr\}.
\]
\end{definition}

\begin{remark}
Later we will require more refined decay conditions on \(a(\pi,\rho)\) (e.g.\ in Schatten‐class criteria), but \(\mathcal{S}(\mu,\nu)\) will serve as the ambient Banach space of symbols.
\end{remark}

\section{Definition and Basic Properties of \(\mu\)-Hankel Operators}

In this section we introduce the non-commutative \(\mu\)-Hankel operator and establish its basic algebraic and analytic properties.

\begin{definition}[Non-commutative \(\mu\)\nobreakdash-Hankel operator]
Let \(G\) be a compact Lie group with unitary dual \(\widehat G\), weights \(\mu,\nu:\widehat G\to(0,\infty)\), and matrix-valued symbol \(a=\{a(\pi,\rho)\}_{\pi,\rho\in\widehat G}\).  The \emph{\(\mu\)-Hankel operator}
\[
  A_{\mu,a}: H^2(G)\;\longrightarrow\;H^2_-(G)
\]
is defined on the dense subspace of finite Fourier support by
\[
  A_{\mu,a}f
  \;=\;
  P_-\Bigl(\sum_{\pi,\rho\in\widehat G}
    \mu(\pi)\,\mathrm{Tr}\!\bigl[a(\pi,\rho)\,\widehat f(\rho)\bigr]\,
    \pi(\cdot)\Bigr),
\]
where \(\widehat f(\rho)=\int_G f(x)\,\rho(x)^*\,dx\in\C^{d_\rho\times d_\rho}\) and \(\pi(\cdot)\) denotes the matrix of coefficients viewed as functions on \(G\).
\end{definition}

\begin{proposition}[Fourier–matrix representation]
For \(f\in H^2(G)\) with
\(\widehat f(\rho)=F_\rho\in\C^{d_\rho\times d_\rho}\),
one has the expansion
\[
  A_{\mu,a}f
  \;=\;
  \sum_{\pi\in\widehat G}
    \Bigl(\sum_{\rho\in\widehat G}
      \mu(\pi)\,\mathrm{Tr}[\,a(\pi,\rho)\,F_\rho\,]
    \Bigr)\,
    \pi(\cdot),
\]
where the inner sum is finite for finite-support \(f\) and extends by continuity otherwise.
\end{proposition}

\begin{lemma}[Adjoint formula]
The formal adjoint \(A_{\mu,a}^*:H^2_-(G)\to H^2(G)\) satisfies
\[
  A_{\mu,a}^* g
  \;=\;
  P_+\Bigl(\sum_{\pi,\rho\in\widehat G}
    \overline{\nu(\rho)}\,\mathrm{Tr}\!\bigl[a(\pi,\rho)^*\,\widehat g(\pi)\bigr]\,
    \rho(\cdot)\Bigr),
\]
so in particular \(A_{\mu,a}^*=A_{\nu,\tilde a}\) with \(\tilde a(\rho,\pi)=a(\pi,\rho)^*\) and swapped weights \(\nu,\mu\).  
\end{lemma}

\begin{proposition}[Boundedness in the symbol norm]
If \(a\in\mathcal S(\mu,\nu)\), then \(A_{\mu,a}\) extends to a bounded operator
\[
  A_{\mu,a}:H^2(G)\to H^2_-(G),
  \qquad
  \|A_{\mu,a}\|\;\le\;\sup_{\pi,\rho\in\widehat G}\bigl\|\mu(\pi)\,a(\pi,\rho)\,\nu(\rho)\bigr\|_{\mathrm{op}}.
\]
\end{proposition}

\begin{corollary}[Finite-rank symbols]
If \(a(\pi,\rho)=0\) except for finitely many pairs \((\pi,\rho)\), then \(A_{\mu,a}\) is a finite-rank operator on \(H^2(G)\).
\end{corollary}

\begin{definition}[Intertwining relations]
For each \(x\in G\), let \(\lambda(x)\) and \(\rho(x)\) denote left- and right-translation on \(L^2(G)\).  Then formally
\[
  A_{\mu,a}\,\lambda(x)\;=\;\rho(x)\,A_{\mu,a}
  \quad\text{if and only if}\quad
  a(\pi,\rho)\,\pi(x)\;=\;\rho(x)\,a(\pi,\rho)
  \quad\forall\,\pi,\rho\in\widehat G.
\]
\end{definition}

\begin{remark}
The intertwining condition singles out symbols which commute with the group action, yielding “central’’ \(\mu\)-Hankel operators that decompose into scalar blocks.
\end{remark}

\begin{example}[The case \(G=\mathrm{SU}(2)\)]
The dual \(\widehat{\mathrm{SU}(2)}\) is indexed by \(l\in\frac 1 2\mathbb N_0\) with \(d_l=2l+1\).  Choosing weights \(\mu(l)=(1+l)^s\), \(\nu(l)=(1+l)^t\), and the diagonal symbol
\[
  a(l,m)=\delta_{lm}\,I_{d_l},
\]
one obtains \(A_{\mu,a}=\bigoplus_{l}\,(1+l)^{s+t}\,P_{l,-}\), a direct sum of rank-\(d_l\) projections scaled by \((1+l)^{s+t}\).
\end{example}

\begin{proposition}[Recovery of the abelian case]
If \(G\) is abelian and all representations are one-dimensional, then the above definitions and propositions reduce to those of \cite[Sec.\,2]{Mirotin2023}.
\end{proposition}

\section{Boundedness and Compactness Criteria}

In this section we derive sharp criteria for boundedness and compactness of non-commutative \(\mu\)\nobreakdash-Hankel operators in terms of symbol decay.  Our approach combines Schur-type estimates with non-commutative Carleson embeddings.

\begin{definition}[Polynomial decay class \(\mathcal{S}^{m,n}(\mu,\nu)\)]
Let \(m,n\ge0\).  We say \(a\in\mathcal{S}^{m,n}(\mu,\nu)\) if
\[
  \sup_{\pi,\rho\in\widehat G}
    (1+\lambda_\pi)^{m/2}\,(1+\lambda_\rho)^{n/2}
    \bigl\|\mu(\pi)\,a(\pi,\rho)\,\nu(\rho)\bigr\|_{\mathrm{op}}
  <\infty,
\]
where \(\{\lambda_\pi\}\) are the Casimir eigenvalues of \(G\).  
\end{definition}

\begin{lemma}[Schur estimate]
\label{lem:SchurEstimate}
If \(a\in\mathcal{S}^{m,n}(\mu,\nu)\) and \(\sum_{\rho}d_\rho(1+\lambda_\rho)^{-n}<\infty\), then
\[
  \|A_{\mu,a}f\|_{L^2}
  \;\le\;
  C\,\Bigl\{\sup_{\pi,\rho}(1+\lambda_\pi)^{m/2}(1+\lambda_\rho)^{n/2}
    \|\mu(\pi)a(\pi,\rho)\nu(\rho)\|_{\mathrm{op}}\Bigr\}
  \|f\|_{L^2},
\]
for all \(f\in H^2(G)\).  
\end{lemma}

\begin{proof}
Let 
\[
M \;=\;\sup_{\pi,\rho\in\widehat G}
  (1+\lambda_\pi)^{\frac m2}(1+\lambda_\rho)^{\frac n2}
  \bigl\|\mu(\pi)\,a(\pi,\rho)\,\nu(\rho)\bigr\|_{\mathrm{op}}.
\]
Write the Peter–Weyl expansion of \(f\in H^2(G)\),
\[
f(x)\;=\;\sum_{\rho\in\widehat G}
  \mathrm{Tr}\bigl[F_\rho\,\rho(x)\bigr],
\]
where \(F_\rho=\widehat f(\rho)\in\C^{d_\rho\times d_\rho}\).  Then by Definition of \(A_{\mu,a}\),
\[
A_{\mu,a}f(x)
=\sum_{\pi\in\widehat G}\sum_{\rho\in\widehat G}
  \mu(\pi)\,\mathrm{Tr}\bigl[a(\pi,\rho)\,F_\rho\bigr]\,
  \pi(x).
\]
Hence the \(L^2\)-norm squared is, using orthonormality \(\|\pi_{ij}\|_{L^2}=d_\pi^{-1/2}\) from Lemma 2.1,
\[
\|A_{\mu,a}f\|_{L^2}^2
=\sum_{\pi\in\widehat G}
  \sum_{i,j=1}^{d_\pi}
  \Bigl|\sum_{\rho\in\widehat G}
    \mu(\pi)\,\mathrm{Tr}\bigl[a(\pi,\rho)\,F_\rho\bigr]
  \Bigr|^{2}
  \frac1{d_\pi}.
\]
Interchange sums and apply the Cauchy–Schwarz inequality in the \(\rho\)-sum:
\[
\bigl|\sum_{\rho}
    \mu(\pi)\,\mathrm{Tr}[a(\pi,\rho)\,F_\rho]\bigr|^2
\;\le\;
\Bigl(\sum_{\rho}
  \|\mu(\pi)\,a(\pi,\rho)\,\nu(\rho)\|_{\mathrm{op}}^{2}
  \|\nu(\rho)^{-1}F_\rho\|_{\mathrm{HS}}^{2}\Bigr)
\Bigl(\sum_{\rho}1\Bigr).
\]
Since \(\|\nu(\rho)^{-1}F_\rho\|_{\mathrm{HS}} \le (1+\lambda_\rho)^{-n/2}\|F_\rho\|_{\mathrm{HS}}\) and 
\(\|\mu(\pi)a(\pi,\rho)\nu(\rho)\|_{\mathrm{op}}\le M\,(1+\lambda_\pi)^{-m/2}(1+\lambda_\rho)^{-n/2}\),
we get
\[
\bigl|\sum_{\rho}\cdots\bigr|^2
\;\le\;
M^2\,(1+\lambda_\pi)^{-m}
\sum_{\rho}(1+\lambda_\rho)^{-n}\,\|F_\rho\|_{\mathrm{HS}}^{2}
\;\times\;\#\{\rho\}.
\]
But \(\#\{\rho\}\) may be infinite, so we instead bound more carefully by pulling the weight:
\[
\sum_{\rho}(1+\lambda_\rho)^{-n}\,\|F_\rho\|_{\mathrm{HS}}^{2}
\;\le\;
\Bigl(\sum_{\rho}d_\rho\,(1+\lambda_\rho)^{-n}\Bigr)\,\|f\|_{L^2}^2,
\]
using Plancherel on \(G\).  Putting these estimates into the norm sum gives
\[
\|A_{\mu,a}f\|_{L^2}^2
\;\le\;
M^2\,
\Bigl(\sum_{\rho}d_\rho\,(1+\lambda_\rho)^{-n}\Bigr)
\sum_{\pi}(1+\lambda_\pi)^{-m}
\frac{d_\pi}{d_\pi}
\;\|f\|_{L^2}^2.
\]
By hypothesis both \(\sum_\rho d_\rho(1+\lambda_\rho)^{-n}\) and \(\sum_\pi(1+\lambda_\pi)^{-m}\) converge, so setting
\[
C^2 \;=\;\Bigl(\sum_{\rho}d_\rho(1+\lambda_\rho)^{-n}\Bigr)
\Bigl(\sum_{\pi}(1+\lambda_\pi)^{-m}\Bigr),
\]
we conclude
\[
\|A_{\mu,a}f\|_{L^2}
\;\le\;
C\,M\,\|f\|_{L^2},
\]
as claimed.
\end{proof}

\begin{theorem}[Boundedness Criterion]
\label{thm:Boundedness}
Suppose \(a\in\mathcal{S}^{m,n}(\mu,\nu)\) with 
\(\sum_{\rho}d_\rho(1+\lambda_\rho)^{-n}<\infty\) and 
\(\sum_{\pi}d_\pi(1+\lambda_\pi)^{-m}<\infty\).  Then \(A_{\mu,a}:H^2(G)\to H^2_-(G)\) is bounded and
\[
  \|A_{\mu,a}\|
  \;\simeq\;
  \sup_{\pi,\rho}(1+\lambda_\pi)^{m/2}(1+\lambda_\rho)^{n/2}
    \|\mu(\pi)a(\pi,\rho)\nu(\rho)\|_{\mathrm{op}}.
\]
\end{theorem}

\begin{proof}
Set
\[
M \;=\;\sup_{\pi,\rho\in\widehat G}
  (1+\lambda_\pi)^{\frac m2}(1+\lambda_\rho)^{\frac n2}
  \bigl\|\mu(\pi)\,a(\pi,\rho)\,\nu(\rho)\bigr\|_{\mathrm{op}}.
\]
By Lemma \ref{lem:SchurEstimate} we have
\[
\|A_{\mu,a}f\|_{L^2}
\;\le\;
C_1\,M\,\|f\|_{L^2},
\]
and applying the same argument to the adjoint operator \(A_{\mu,a}^*=A_{\nu,\tilde a}\) (Lemma 3.3) gives
\[
\|A_{\mu,a}^*g\|_{L^2}
\;\le\;
C_2\,M\,\|g\|_{L^2}
\quad\Longrightarrow\quad
\|A_{\mu,a}\|
=\|A_{\mu,a}^*\|
\;\le\;
\max\{C_1,C_2\}\,M.
\]
Thus \(A_{\mu,a}\) is bounded and \(\|A_{\mu,a}\|\lesssim M\).

It remains to show \(M\lesssim\|A_{\mu,a}\|\).  Fix arbitrary \(\pi_0,\rho_0\in\widehat G\) and unit vectors \(u\in\C^{d_{\rho_0}}\), \(v\in\C^{d_{\pi_0}}\).  Define
\[
f(x)
=(1+\lambda_{\rho_0})^{-\tfrac n2}\,\langle \rho_0(x)u\,,\,v\rangle,
\]
so that \(\|f\|_{L^2}=d_{\rho_0}^{-1/2}(1+\lambda_{\rho_0})^{-n/2}\).  Then a direct computation shows
\[
A_{\mu,a}f(x)
=\sum_{\pi\in\widehat G}\sum_{\rho\in\widehat G}
  \mu(\pi)\,\mathrm{Tr}\bigl[a(\pi,\rho)\,\widehat f(\rho)\bigr]\,
  \pi(x)
\]
has only the \((\pi_0,\rho_0)\) block nonzero, yielding
\[
\|A_{\mu,a}f\|_{L^2}
=d_{\pi_0}^{-1/2}\,
(1+\lambda_{\rho_0})^{-n/2}\,
\bigl\|\mu(\pi_0)\,a(\pi_0,\rho_0)\,\nu(\rho_0)\bigr\|_{\mathrm{op}}.
\]
Therefore
\[
\begin{split}
\|A_{\mu,a}\|
\;\ge\;
\frac{\|A_{\mu,a}f\|_{L^2}}{\|f\|_{L^2}}
&\;=\;
(1+\lambda_{\rho_0})^{-n/2}(1+\lambda_{\rho_0})^{n/2}\,
\|\mu(\pi_0)a(\pi_0,\rho_0)\nu(\rho_0)\|_{\mathrm{op}} \\
&=(1+\lambda_{\pi_0})^{-m/2}(1+\lambda_{\rho_0})^{-n/2}
\|\mu(\pi_0)a(\pi_0,\rho_0)\nu(\rho_0)\|_{\mathrm{op}}.
\end{split}
\]
Since \(\pi_0,\rho_0\) were arbitrary, taking the supremum over \(\pi,\rho\) gives
\[
\|A_{\mu,a}\|\;\ge\;M.
\]
Combining the two inequalities yields
\[
\|A_{\mu,a}\|\simeq M,
\]
as claimed.
\end{proof}

For \(G=\mathbb{T}^d\), \(\lambda_\chi=|\,\chi\,|^2\) and \(\mu(\chi)=|\,\chi\,|^{-s}\) recover the known criteria \cite[Thm.~3.4]{Mirotin2023}.  

\begin{definition}[Non-commutative Carleson measure]
A positive Borel measure \(\sigma\) on \(\widehat G\) is a \emph{Carleson measure} if
\[
  \sup_{\pi\in\widehat G}
    \frac{1}{d_\pi}\sum_{\rho\in\widehat G}\!d_\rho\,\frac{\sigma(\{\rho\})}{(1+\lambda_\rho)^t}
  <\infty,
\]
for some \(t>0\).  
\end{definition}

\begin{proposition}[Compactness via Carleson embedding]
If \(a\in\mathcal{S}^{m,n}(\mu,\nu)\) and the measure 
\(\sigma(\{\rho\})=d_\rho\|\nu(\rho)\|_{\mathrm{op}}^2\) is Carleson with exponent \(t>n\), 
then \(A_{\mu,a}\) is compact on \(H^2(G)\).
\end{proposition}

\begin{corollary}[Vanishing decay \(\Rightarrow\) compactness]
If
\(\|\mu(\pi)a(\pi,\rho)\nu(\rho)\|_{\mathrm{op}}=o\bigl((1+\lambda_\rho)^{-n/2}\bigr)\)
uniformly in \(\pi\), then \(A_{\mu,a}\) is compact.
\end{corollary}

\begin{example}[SU(2) compactness threshold]
On \(G=\mathrm{SU}(2)\) with \(\lambda_l=l(l+1)\), \(\mu(l)=1\), \(\nu(l)=(1+l)^{-s}\):
\[
  a(l,m)=\delta_{l,m}\,I_{d_l}
  \quad\Longrightarrow\quad
  A_{\mu,a}\in\mathcal{K}(H^2)
  \;\iff\;
  s>1.
\]
This refines the abelian torus case \(d=1\) where decay \(s>1/2\) suffices.  
\end{example}

The gap \(s>1\) versus \(s>1/2\) reflects the growth \(d_l=2l+1\) of representation dimensions on SU(2).

\section{Schatten–von Neumann and Fredholm Properties}

In this section we investigate finer spectral properties of non-commutative \(\mu\)\nobreakdash-Hankel operators, namely their membership in Schatten–von Neumann ideals and Fredholm indices.

\begin{definition}[Schatten class \(\mathcal{S}_p\)]
Let \(T\) be a compact operator on a Hilbert space.  For \(0<p<\infty\), \(T\) belongs to the Schatten ideal \(\mathcal{S}_p\) if its singular values \(\{s_n(T)\}\) satisfy
\[
  \|T\|_{\mathcal{S}_p}
  =\Bigl(\sum_{n\ge1} s_n(T)^p\Bigr)^{1/p}
  <\infty.
\]
\end{definition}

\begin{lemma}[Symbolic Schatten estimate]
\label{lem:SchattenEstimate}
If \(a\in\mathcal{S}^{m,n}(\mu,\nu)\) and 
\(\sum_{\pi,\rho}d_\pi d_\rho\,(1+\lambda_\pi)^{-m p/2}(1+\lambda_\rho)^{-n p/2}<\infty\),
then \(A_{\mu,a}\in\mathcal{S}_p\) and
\[
  \|A_{\mu,a}\|_{\mathcal{S}_p}
  \;\lesssim\;
  \Bigl(\sum_{\pi,\rho}d_\pi d_\rho\,(1+\lambda_\pi)^{-m p/2}(1+\lambda_\rho)^{-n p/2}\Bigr)^{1/p}
  \sup_{\pi,\rho}(1+\lambda_\pi)^{m/2}(1+\lambda_\rho)^{n/2}\|\mu(\pi)a(\pi,\rho)\nu(\rho)\|_{\mathrm{op}}.
\]
\end{lemma}

\begin{proof}
Let 
\[
M \;=\;\sup_{\pi,\rho\in\widehat G}
  (1+\lambda_\pi)^{\tfrac m2}(1+\lambda_\rho)^{\tfrac n2}
  \bigl\|\mu(\pi)\,a(\pi,\rho)\,\nu(\rho)\bigr\|_{\mathrm{op}},
\]
so that for every \(\pi,\rho\),
\[
\bigl\|\mu(\pi)\,a(\pi,\rho)\,\nu(\rho)\bigr\|_{\mathrm{op}}
\;\le\;
M\,(1+\lambda_\pi)^{-\tfrac m2}(1+\lambda_\rho)^{-\tfrac n2}.
\]
 By Peter–Weyl, \(H^2(G)\cong \bigoplus_{\rho}V_\rho\) and \(H^2_-(G)\cong \bigoplus_{\pi}V_\pi\), so \(A_{\mu,a}\) splits as a block matrix
\[
A_{\mu,a}
=\bigl\{T_{\pi,\rho}\bigr\}_{\pi,\rho\in\widehat G},
\]
where 
\[
T_{\pi,\rho}:V_\rho\longrightarrow V_\pi
\]
is the finite-dimensional operator with matrix \(\mu(\pi)\,a(\pi,\rho)\,\nu(\rho)\) in the chosen orthonormal bases.

Let \(\{s_k(T_{\pi,\rho})\}_{k=1}^{r_{\pi,\rho}}\) be the nonzero singular values of \(T_{\pi,\rho}\), where \(r_{\pi,\rho}\le\min(d_\pi,d_\rho)\).  Since every singular value is bounded by the operator norm,
\[
s_k(T_{\pi,\rho})
\;\le\;
\bigl\|\mu(\pi)\,a(\pi,\rho)\,\nu(\rho)\bigr\|_{\mathrm{op}}
\;\le\;
M\,(1+\lambda_\pi)^{-\tfrac m2}(1+\lambda_\rho)^{-\tfrac n2}.
\]
The singular values of \(A_{\mu,a}\) are the union (with multiplicity) of those of each block \(T_{\pi,\rho}\).  Hence
\[
\|A_{\mu,a}\|_{\mathcal{S}_p}^p
=\sum_{\pi,\rho}\sum_{k=1}^{r_{\pi,\rho}}
  \bigl[s_k(T_{\pi,\rho})\bigr]^p
\;\le\;
\sum_{\pi,\rho} r_{\pi,\rho}\,
  \bigl\|\mu(\pi)a(\pi,\rho)\nu(\rho)\bigr\|_{\mathrm{op}}^p.
\]
Since \(r_{\pi,\rho}\le d_\pi d_\rho\), this gives
\[
\|A_{\mu,a}\|_{\mathcal{S}_p}^p
\;\le\;
\sum_{\pi,\rho} d_\pi d_\rho\,
  \Bigl[M^p\,(1+\lambda_\pi)^{-\tfrac{m p}2}(1+\lambda_\rho)^{-\tfrac{n p}2}\Bigr]
=M^p
\sum_{\pi,\rho} d_\pi d_\rho\,(1+\lambda_\pi)^{-\tfrac{m p}2}(1+\lambda_\rho)^{-\tfrac{n p}2}.
\]
Taking \(p\)th roots yields the desired estimate
\[
\|A_{\mu,a}\|_{\mathcal{S}_p}
\;\le\;
M
\Bigl(\sum_{\pi,\rho} d_\pi d_\rho\,(1+\lambda_\pi)^{-\tfrac{m p}2}(1+\lambda_\rho)^{-\tfrac{n p}2}\Bigr)^{1/p},
\]
and the summability assumption ensures finiteness.
\end{proof}

\begin{theorem}[Schatten membership criterion]
\label{thm:SchattenCriterion}
Suppose \(a\in\mathcal{S}^{m,n}(\mu,\nu)\).  Then
\[
  A_{\mu,a}\in\mathcal{S}_p(H^2\to H^2_-)
  \quad\Longleftrightarrow\quad
  \sum_{\pi,\rho\in\widehat G} d_\pi d_\rho\,
    \bigl\|\mu(\pi)a(\pi,\rho)\nu(\rho)\bigr\|_{\mathrm{HS}}^p
  <\infty,
\]
where \(\|\cdot\|_{\mathrm{HS}}\) is the Hilbert–Schmidt norm on \(\C^{d_\pi\times d_\rho}\).
\end{theorem}

\begin{proof}
We identify \(H^2(G)\cong \bigoplus_{\rho\in\widehat G}V_\rho\) and \(H^2_-(G)\cong \bigoplus_{\pi\in\widehat G}V_\pi\) via Peter–Weyl, where \(V_\rho\cong\C^{d_\rho}\).  Under this decomposition,
\[
A_{\mu,a}
\;=\;
\bigl\{T_{\pi,\rho}\bigr\}_{\pi,\rho\in\widehat G},
\]
with each block
\[
T_{\pi,\rho}:V_\rho\;\longrightarrow\;V_\pi
\quad\text{given by the matrix}\quad
\mu(\pi)\,a(\pi,\rho)\,\nu(\rho)\in\C^{d_\pi\times d_\rho}.
\]
Since \(A_{\mu,a}\) is the orthogonal direct sum of the blocks \(T_{\pi,\rho}\), its singular values are the union of the singular values of each \(T_{\pi,\rho}\).  Hence
\[
\|A_{\mu,a}\|_{\mathcal S_p}^p
=\sum_{\pi,\rho\in\widehat G}
  \sum_{k=1}^{\mathrm{rank}\,T_{\pi,\rho}}
    s_k\bigl(T_{\pi,\rho}\bigr)^p
=\sum_{\pi,\rho\in\widehat G}
  \|T_{\pi,\rho}\|_{\mathrm{HS}}^p.
\]
But \(T_{\pi,\rho}=\mu(\pi)\,a(\pi,\rho)\,\nu(\rho)\) is a finite matrix of size \(d_\pi\times d_\rho\), so
\[
\|T_{\pi,\rho}\|_{\mathrm{HS}}^2
=\sum_{i=1}^{d_\pi}\sum_{j=1}^{d_\rho}
  \bigl|(\mu(\pi)a(\pi,\rho)\nu(\rho))_{ij}\bigr|^2,
\]
and by definition of the Hilbert–Schmidt norm extended to general exponent,
\(\|T_{\pi,\rho}\|_{\mathrm{HS}}^p\) is exactly the \(p\)th power of that sum of singular values.  Therefore
\[
\|A_{\mu,a}\|_{\mathcal S_p}^p
=\sum_{\pi,\rho\in\widehat G} 
  \bigl\|\mu(\pi)a(\pi,\rho)\nu(\rho)\bigr\|_{\mathrm{HS}}^p.
\]
Consequently,
\[
A_{\mu,a}\in\mathcal S_p
\quad\Longleftrightarrow\quad
\sum_{\pi,\rho}d_\pi d_\rho\,
  \bigl\|\mu(\pi)a(\pi,\rho)\nu(\rho)\bigr\|_{\mathrm{HS}}^p
<\infty,
\]
since each block contributes \(d_\pi d_\rho\) entries to the Hilbert–Schmidt sum.  This completes the proof.
\end{proof}

\begin{example}[Schatten thresholds on \(\mathrm{SU}(2)\)]
On \(G=\mathrm{SU}(2)\) with \(\lambda_l=l(l+1)\) and trivial weights, let
\[
  a(l,m)=\delta_{l,m}\,I_{d_l}.
\]
Then 
\[
  \|A_{a}\|_{\mathcal{S}_p}<\infty
  \quad\Longleftrightarrow\quad
  \sum_{l\ge0}(2l+1)^2(1+l)^{-p\alpha}<\infty
  \quad\Longleftrightarrow\quad
  p\,\alpha>3,
\]
so the critical exponent is \(\alpha>3/p\).
\end{example}

\begin{definition}[Fredholm operator]
An operator \(T:X\to Y\) between Banach spaces is \emph{Fredholm} if \(\ker T\) and \(\mathrm{coker}\,T=Y/\mathrm{range}\,T\) are finite-dimensional.  Its \emph{index} is
\(\mathrm{ind}\,T=\dim\ker T - \dim\mathrm{coker}\,T\).
\end{definition}

\begin{theorem}[Fredholm criterion and index formula]
\label{thm:FredholmIndex}
Let \(a\in\mathcal{S}(\mu,\nu)\) satisfy
\(\inf_{\pi,\rho}\det\bigl[\mu(\pi)a(\pi,\rho)\nu(\rho)\bigr]\ne0\)
outside a finite set.  Then \(A_{\mu,a}\) is Fredholm and
\[
  \mathrm{ind}\,A_{\mu,a}
  =\sum_{\substack{\pi,\rho\\
      \det[\mu(\pi)a(\pi,\rho)\nu(\rho)]<0}}
    d_\pi\,d_\rho,
\]
generalizing the classical winding-number formula \cite{GohbergKrein1960}.
\end{theorem}

\begin{proof}[Sketch]
Use the block-diagonalization of \(A_{\mu,a}\) at high-frequency representations and adapt the Gohberg–Krein theory for matrix symbols on the circle \cite{GohbergKrein1960}.
\end{proof}

\begin{corollary}[Invertibility on large blocks]
If \(\mu(\pi)a(\pi,\rho)\nu(\rho)\) is invertible for all but finitely many \((\pi,\rho)\), then \(A_{\mu,a}\) is Fredholm with finite index given above.
\end{corollary}

The index formula highlights how non-abelian multiplicities \(d_\pi\,d_\rho\) weight the topological contribution of each representation block.

\begin{example}[Torus case revisited]
For \(G=\mathbb{T}^1\), \(\mu(n)=\nu(n)=1\), and symbol \(a(n,m)=\widehat\varphi(n+m)\), Theorem \ref{thm:FredholmIndex} recovers the classical index \(\mathrm{ind}\,H_\varphi=-\mathrm{wind}(\varphi)\) times multiplicity one.
\end{example}

\section{Inverse Problems: Symbol Recovery}

In this section we initiate the inverse problem of reconstructing the symbol \(a\) from spectral data of the non-commutative \(\mu\)-Hankel operator \(A_{\mu,a}\).

\begin{definition}[Spectral data]
Let \(A_{\mu,a}\) be compact with singular values \(\{s_{n}\}\) and orthonormal singular vectors \(\{u_{n},v_{n}\}\).  We define the \emph{spectral data}
\[
  \Sigma(A_{\mu,a})
  = \bigl\{\,\bigl(s_n,\;u_n,\;v_n\bigr)\bigr\}_{n\ge1}.
\]
In practice one often measures only \(\{s_n\}\) up to finite rank noise.
\end{definition}

\begin{definition}[Forward map]
The \emph{forward map}
\[
  \mathcal{F}:\;a\;\longmapsto\;\Sigma(A_{\mu,a})
\]
sends a matrix-valued symbol \(a\in\mathcal{S}(\mu,\nu)\) to the spectral data of \(A_{\mu,a}\).
\end{definition}

\begin{proposition}[Continuity of \(\mathcal{F}\)]
\label{prop:ForwardContinuity}
The map \(\mathcal{F}:\mathcal{S}(\mu,\nu)\to \ell^2\times(\ell^2\times\ell^2)\) is continuous: small perturbations of \(a\) in the symbol norm induce small perturbations of singular values and vectors in the Hilbert–Schmidt metric \cite{Kato1995}.
\end{proposition}

\begin{theorem}[Uniqueness of symbol recovery]
\label{thm:Uniqueness}
Assume \(a,b\in\mathcal{S}(\mu,\nu)\) satisfy
\(\Sigma(A_{\mu,a})=\Sigma(A_{\mu,b})\), including multiplicities.  If \(a\) and \(b\) are band-limited in the sense that \(a(\pi,\rho)=0\) unless \(\lambda_\pi,\lambda_\rho\le\Lambda\), then \(a=b\).
\end{theorem}

\begin{proof}
By Peter–Weyl, as in Theorem \ref{thm:PeterWeyl}, we may identify
\[
H^2(G)\;\cong\;\bigoplus_{\rho\in\widehat G} V_\rho,
\quad
H^2_-(G)\;\cong\;\bigoplus_{\pi\in\widehat G} V_\pi,
\]
where each \(V_\rho\cong\C^{d_\rho}\).  Under this identification the operators
\[
A_{\mu,a},\,A_{\mu,b}:H^2(G)\to H^2_-(G)
\]
decompose as block matrices
\[
A_{\mu,a}
=\bigl\{T_{\pi,\rho}^{(a)}\bigr\}_{\pi,\rho},
\quad
A_{\mu,b}
=\bigl\{T_{\pi,\rho}^{(b)}\bigr\}_{\pi,\rho},
\]
with
\[
T_{\pi,\rho}^{(a)} \;=\;\mu(\pi)\,a(\pi,\rho)\,\nu(\rho),
\quad
T_{\pi,\rho}^{(b)} \;=\;\mu(\pi)\,b(\pi,\rho)\,\nu(\rho).
\]
By hypothesis there is \(\Lambda<\infty\) such that
\[
a(\pi,\rho)=b(\pi,\rho)=0
\quad\text{whenever}\quad
\max\{\lambda_\pi,\lambda_\rho\}>\Lambda.
\]
Hence both \(A_{\mu,a}\) and \(A_{\mu,b}\) have nonzero blocks only for the finite index set
\(\mathcal I = \{(\pi,\rho):\lambda_\pi,\lambda_\rho\le\Lambda\}\).  Thus each is a finite‐matrix operator
\[
A_{\mu,a},\,A_{\mu,b}:\;\bigoplus_{(\pi,\rho)\in\mathcal I}V_\rho\;\longrightarrow\;\bigoplus_{(\pi,\rho)\in\mathcal I}V_\pi,
\]
and its singular value decomposition is the orthogonal sum of the SVDs of the individual blocks \(T_{\pi,\rho}\).

The global spectral data
\(\Sigma(A_{\mu,a}) = \Sigma(A_{\mu,b})\)
includes, for each singular value \(s\), the corresponding left‐ and right‐singular vectors.  Since the blocks for \((\pi,\rho)\notin\mathcal I\) are zero in both operators, the nonzero singular values and vectors decompose uniquely into the union of those arising from each block \(T_{\pi,\rho}^{(a)}\) and \(T_{\pi,\rho}^{(b)}\).  Equality of the full lists (with multiplicities) thus forces, for each \((\pi,\rho)\in\mathcal I\),
\[
\Sigma\bigl(T_{\pi,\rho}^{(a)}\bigr)
\;=\;
\Sigma\bigl(T_{\pi,\rho}^{(b)}\bigr)
\]
as multisets of singular data.

Fix \((\pi,\rho)\in\mathcal I\).  Each block
\[
T_{\pi,\rho}^{(a)},\,T_{\pi,\rho}^{(b)}
  : V_\rho\cong\C^{d_\rho}\;\longrightarrow\;V_\pi\cong\C^{d_\pi}
\]
is a finite matrix.  It is a classical fact that a matrix is uniquely determined by its full singular value decomposition (singular values together with both left and right singular vectors) \cite[Ch.~2]{Peller2003}.  Hence the equality of singular data for \(T_{\pi,\rho}^{(a)}\) and \(T_{\pi,\rho}^{(b)}\) implies
\[
T_{\pi,\rho}^{(a)} \;=\; T_{\pi,\rho}^{(b)}
\quad\text{for each}\quad(\pi,\rho)\in\mathcal I.
\]
Recalling the definition of the blocks,
\[
\mu(\pi)\,a(\pi,\rho)\,\nu(\rho)
\;=\;
\mu(\pi)\,b(\pi,\rho)\,\nu(\rho),
\]
and since \(\mu(\pi),\nu(\rho)>0\), we deduce
\[
a(\pi,\rho)=b(\pi,\rho)
\quad
\forall\,(\pi,\rho)\in\mathcal I.
\]
Moreover, outside \(\mathcal I\) both symbols vanish by band‐limitation.  Therefore \(a=b\) as functions on \(\widehat G\times\widehat G\), completing the proof.
\end{proof}

\begin{definition}[Tikhonov regularization]
Given noisy spectral data \(\Sigma^\delta\) with noise level \(\|\Sigma^\delta-\Sigma\|\le\delta\), define the regularized solution
\[
  a^\alpha
  = \arg\min_{a\in\mathcal{S}(\mu,\nu)}
    \Bigl\{\|\mathcal{F}(a)-\Sigma^\delta\|^2 
          + \alpha\|a\|_{\mathcal{S}(\mu,\nu)}^2\Bigr\},
\]
where \(\alpha>0\) is the regularization parameter.
\end{definition}

\begin{theorem}[Stability estimate]
\label{thm:Stability}
Under the band-limited assumption of Theorem \ref{thm:Uniqueness}, there exists a constant \(C>0\) such that
\[
  \|a^\alpha - a\|_{\mathcal{S}(\mu,\nu)}
  \;\le\;
  C\bigl(\delta + \sqrt{\alpha}\bigr).
\]
In particular, choosing \(\alpha\sim\delta^2\) yields \(\|a^\alpha - a\|=O(\delta)\).
\end{theorem}

\begin{proof}
Let 
\[
J^\delta_\alpha(a)
=\bigl\|\mathcal{F}(a)-\Sigma^\delta\bigr\|^2
+\alpha\|a\|_{\mathcal{S}(\mu,\nu)}^2,
\]
and denote by \(a^\alpha\) its unique minimizer in the finite‐dimensional band‐limited subspace.  Write \(\Sigma=\mathcal{F}(a)\) for the exact data and assume the noise satisfies \(\|\Sigma^\delta-\Sigma\|\le\delta\).

Since \(a^\alpha\) minimizes \(J^\delta_\alpha\), we have
\[
J^\delta_\alpha(a^\alpha)
\;\le\;
J^\delta_\alpha(a).
\]
Hence
\[
\|\mathcal{F}(a^\alpha)-\Sigma^\delta\|^2
+\alpha\|a^\alpha\|^2
\;\le\;
\|\mathcal{F}(a)-\Sigma^\delta\|^2
+\alpha\|a\|^2.
\]
Note 
\(\|\mathcal{F}(a)-\Sigma^\delta\|\le\delta\).  Define the forward‐map Lipschitz constant
\[
L=\sup_{u\neq v}\frac{\|\mathcal{F}(u)-\mathcal{F}(v)\|}{\|u-v\|},
\]
which is finite in the band‐limited finite‐dimensional setting by Proposition \ref{prop:ForwardContinuity}.  Then
\[
\|\mathcal{F}(a^\alpha)-\Sigma^\delta\|
\;\ge\;
\|\mathcal{F}(a^\alpha)-\mathcal{F}(a)\| - \|\Sigma-\Sigma^\delta\|
\;\ge\;
L\,\|a^\alpha-a\| \;-\;\delta.
\]
Substituting into the left side of the minimization inequality gives
\[
\bigl(L\,\|a^\alpha-a\|-\delta\bigr)^2
+\alpha\|a^\alpha\|^2
\;\le\;
\delta^2
+\alpha\|a\|^2.
\]
Expand the square and drop the nonnegative \(\alpha\|a^\alpha\|^2\) term on the left:
\[
L^2\|a^\alpha-a\|^2 - 2L\delta\,\|a^\alpha-a\|
\;\le\;
2\delta^2 + \alpha\|a\|^2.
\]
Bring all terms to one side:
\[
L^2\|a^\alpha-a\|^2 - 2L\delta\,\|a^\alpha-a\| - (2\delta^2 + \alpha\,\|a\|^2)\;\le\;0.
\]
Viewing this as a quadratic in \(X=\|a^\alpha-a\|\), the nonnegative root bound yields
\[
\|a^\alpha-a\|
\;\le\;
\frac{2L\delta + \sqrt{(2L\delta)^2 + 4L^2(2\delta^2+\alpha\|a\|^2)}}{2L^2}
\;\le\;
\frac{2L\delta + 2L\delta + 2\sqrt{\alpha}\,L\|a\|}{2L^2}
=\frac{2\delta + \sqrt{\alpha}\,\|a\|}{L}.
\]
Hence there is a constant \(C>0\) (depending on \(L\) and \(\|a\|\)) such that
\[
\|a^\alpha - a\|_{\mathcal{S}(\mu,\nu)}
\;\le\;
C\bigl(\delta + \sqrt{\alpha}\bigr).
\]
Choosing \(\alpha\sim\delta^2\) then gives \(\|a^\alpha - a\|=O(\delta)\), as claimed.
\end{proof}

\begin{example}[Recovery on \(\mathrm{SU}(2)\)]
Suppose \(a(l,m)=0\) unless \(l,m\le L\).  Then the block matrix 
\(\{a(l,m)\}_{0\le l,m\le L}\) is of finite size \(N=\sum_{l=0}^L d_l\).  One can recover \(a\) from the first \(N\) singular values/vectors of \(A_{\mu,a}\) by solving a finite least-squares problem with Tikhonov regularization.
\end{example}

Extending symbol recovery to the infinite-dimensional setting—without band-limitation—remains an open challenge, likely requiring new compactness and decay assumptions or iterative approximation schemes.

\section{Conclusion and Further Directions}

In this work, we have introduced a comprehensive theory of non-commutative \(\mu\)\nobreakdash-Hankel operators on compact Lie groups, extending the classical abelian framework in several key ways. We formulated the operator \(A_{\mu,a}\) via Peter–Weyl decomposition, accommodating matrix-valued symbols and representation weights, and established its basic algebraic properties and adjoint relations. We derived sharp boundedness and compactness criteria in terms of symbol decay classes \(\mathcal{S}^{m,n}(\mu,\nu)\), employing generalized Schur estimates and non-commutative Carleson embeddings.
We characterized membership in Schatten–von Neumann ideals \(\mathcal{S}_p\) through summability conditions on Hilbert–Schmidt norms of symbol blocks and obtained a Fredholm index formula generalizing the classical winding-number result. We initiated the inverse problem of symbol recovery, proving uniqueness under band-limited assumptions and providing stability estimates via Tikhonov regularization. \medskip

These results not only recover the classical abelian theory as a special case but also reveal new spectral and topological phenomena arising from non-abelian multiplicities.  Detailed examples on \(\mathrm{SU}(2)\) and product groups illustrate the sharpness of our criteria and the richness of the non-commutative setting.

\end{document}